\documentclass{birkmult}
\usepackage[all]{xy}
\xyoption{poly} \xyoption{graph} \xyoption{arc} \xyoption{web}
\xyoption{2cell}
\usepackage{float}
 \newtheorem{thm}{Theorem}[section]
 
 \newtheorem{lem}[thm]{Lemma}
 
 \theoremstyle{definition}
 
 \theoremstyle{remark}
 
 \newtheorem{ex}{Example}
 \numberwithin{equation}{section}

\begin{document}
%
%
%
%
%
%
%
%
%
\title
 {The Realizable Extension Problem and the Weighted Graph $(K_{3,3},l)$}
\author{Jonathan McLaughlin}

\address{%
School of Mathematics\\
National University of Ireland, Galway\\
University Road, Galway\\
Ireland}

\email{\tt{j.mclaughlin2@nuigalway.ie}}

\subjclass{Primary 55R80; Secondary 51-XX}

\keywords{Moduli spaces, realizations, weighted graphs, configuration space}

\date{June 12, 2010}
\dedicatory{To my parents John and Colette}

\begin{abstract} This note outlines the realizable extension problem for weighted graphs and provides results of a detailed analysis of this problem for the weighted graph $(K_{3,3},l)$. This analysis is then utilized to provide a result relating to the connectedness of the moduli space of planar realizations of $(K_{3,3},l)$. The note culminates with two examples which show that in general, realizability and connectedness results relating to the moduli spaces of weighted cycles which are contained in a larger weighted graph cannot be extended to similar results regarding the moduli space of the larger weighted graph.
\end{abstract}
\maketitle

\parindent=0cm

\section{Introduction}
Given a graph with preassigned edge lengths then a common problem is to determine if this weighted graph can be realized in $\mathbb{E}^{2}$. A \textit{graph} $G$ is a pair $(V_{G},E_{G})$ where $V_{G}$, known as the \textit{vertex set} of $G$, is a finite set, and $E_{G}$, known as the \textit{edge set} of $G$, is a multiset whose elements are elements of $[V_{G}]^{2}$, the set of $2$-element subsets of $V_{G}$. Each edge $\{i,j\}$ is denoted $ij$ in the sequel. In this note, graphs can have parallel edges but not loops. For further detail regarding graph theory, see \cite{D}. A length function on a graph $G$ is a function $l:E_{G}
\rightarrow \mathbb{R}^{\geq 0}$. A weighted graph is a pair $(G,l)$ where $G$ is a graph and $l$ is a length function on $E_{G}$. Given a weighted graph $(G,l)$, then the \textit{configuration space} $C(G,l)$ of $(G,l)$ is defined as  \[ C(G,l) = \{p:V_{G} \rightarrow
\mathbb{E}^2 \;\mid \; d(p(u),p(v)) = l(uv) \text{ for all } uv \in E_{G}\}\]

Each $p$ contained in $C(G,l)$ is called a \textit{realization} of $(G,l)$ and if there exists a realization of $(G,l)$, then the weighted graph $(G,l)$ is said to be \textit{realizable}. Note that in the sequel, and particularly in figures, given a realization $p$ then $p_{\mid v_{i}}$ is denoted $p_{i}$. Given a graph $G$ with vertex set $V_{G}$ then the group $\mathbb{E}^{+}(2)$ of \textit{orientation preserving isometries} of $\mathbb{E}^{2}$ acts on $C(G,l)$ by
\begin{center}
$(\mathbf{g}.p)(v)= \mathbf{g}.(p(v))$ for all $v\in V_{G}$ 
\end{center} 
Given a weighted graph $(G,l)$ and the configuration space $C(G,l)$, then the \textit{moduli space} $M(G,l)$ of $(G,l)$ is the quotient space 
\begin{center}
$M(G,l)=C(G,l)/\mathbb{E}^{+}(2)$
\end{center}

Elements of a moduli space $M(G,l)$ are equivalence classes and so are usually  denoted by $[p]$, however, whenever no confusion can arise, by a slight abuse of notation, the elements of $M(G,l)$ are simply denoted $p$ in the sequel.\\

A subspace of a configuration space which is utilized in the sequel is now described. Given a weighted graph $(G,l)$, the vertices $a$ and $b$ in $V_{G}$ such that $ab \in E_G$ and that $l(ab)>0$, then define   \[C_{a,b}(G,l) = \{ p \in C(G,l)\;\mid\; p(a) = (0,0)\;{\rm and}\; p(b) = (l(ab),0)\}\]

Note that $C_{a,b}(G,l)$ and $C_{b,a}(G,l)$ are different as sets but are homeomorphic topological spaces. Observe that given a weighted graph $(G,l)$ then the space $C_{a,b}(G,l)$ is homeomorphic to the moduli space $M(G,l)$.\\

The \textit{realizability problem} for a weighted graph is the problem of  establishing whether or not there exists a realization of $(G,l)$ and, in general, this problem is hard. Note that this problem is sometimes referred to as the \textit{molecule problem} and for further details on this see \cite{BC} and \cite{He}. One of the simplest weighted graphs for which the realizability problem has been solved is $(K^{4},k)$, where $K^{4}$ is the complete graph on four vertices and this solution is now briefly outlined. Consider $(K^{4},k)$, with vertex set  $V_{K^{4}}=\{v_{1},v_{2},v_{3},v_{4}\}$ and edge set $E_{K^{4}}=\{v_{1}v_{2},v_{1}v_{3},v_{1}v_{4},v_{2}v_{3},v_{2}v_{4}, v_{3}v_{4}\}$. It is assumed throughout this section that the lengths assigned by $k$ are denoted as follows $ k(v_{1}v_{2})=a$, $ k(v_{2}v_{4})=b$, $ k(v_{3}v_{4})=c$, $ k(v_{1}v_{3})=d$, $ k(v_{2}v_{3})=\alpha$ and $ k(v_{1}v_{4})=\beta$. This notation is illustrated in Fig. \ref{K1}. \\

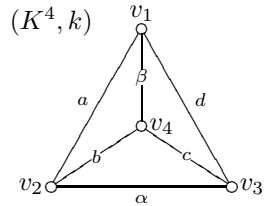
\begin{figure}[b]
\begin{center}
$\begin{xy}
 \POS (12,21) *\cir<2pt>{} ="a1" *+!D{v_{1}} ,
 (0,0) *\cir<2pt>{} ="b1" *+!R{v_{2}},
 (24,0) *\cir<2pt>{} ="c1" *+!L{v_{3}},
 (12,8) *\cir<2pt>{} ="d1" *+!L{v_{4}},
 (0,21) *+!{(K^{4},k)}
 
\POS "a1" \ar@{-}_{a}  "b1",
\POS "b1" \ar@{-}_{\alpha}  "c1",
\POS "c1" \ar@{-}_{d}  "a1",
\POS "d1" \ar@{-}|{\beta}  "a1",
\POS "d1" \ar@{-}|{b}  "b1",
\POS "d1" \ar@{-}|{c}  "c1",

 \end{xy}$ 
 \caption{The weighted graph $(K^{4},k)$ }
\label{K1}
\end{center}
\end{figure}

It is well known, see \cite{AS} for instance, that $(K^{4},k)$ is realizable if and only if all cyclic permutations of the four inequalities $a\leq b+\beta$, $b\leq c+\alpha$, $c\leq d+\beta$ and $d\leq a+\alpha$ are satisfied and equation \ref{eqn1} holds. Note that the determinant contained in equation \ref{eqn1} is known as the {\it Cayley-Menger determinant.}
\begin{equation}
\label{eqn1}
\det \left(\begin{array}{cccccc}
0 & 1 & 1 & 1 &  1 \\
1 & 0 &  a^{2} &  d^{2} &   \beta^{2} \\
1 &  a^{2} & 0 &  \alpha^{2} &   b^{2} \\
1 & d^{2} &  \alpha^{2} & 0 &   c^{2} \\
1 &  \beta^{2} &  b^{2} &  c^{2} & 0
\end{array}\right) =\;0
\end{equation}

The fact that realizability conditions exist for the weighted graph $(K^{4},k)$ appears to be something of a rarity as there does not appear to exist in the literature general realizability conditions, analogous to the $(K^{4},k)$ case for other (non-trivial) weighted graphs. However, one recent development to this end, is a result contained in \cite{JML} (and will appear in \cite{JC}) which gives realizability conditions for weighted graphs where the graph is contained in the class of series-parallel graphs. \\

At this point the focus switches from the \textit{realizability problem} to the following, more tractable, \textit{realizable extension problem}. Given a realizable weighted graph $(H,h)$ where $H\subset G$, then what conditions must an extension of $h$, denoted $l$, satisfy so that $(G,l)$ is realizable. Observe that as every graph has a spanning tree (or spanning forest if the graph is not connected) then it is possible to state the following elementary existence result for such extensions.

\begin{lem}\label{exten2} Given a graph $G$ and a realizable weighted graph $(H,h)$ where $H\subset G$, then it is possible to find an extension of $h$, denoted $l$, such that $(G,l)$ is realizable.
\end{lem}

\section{The Realizable Extension Problem for $(K_{3,3},l)$}

The realizable extension problem is now examined in the case of the weighted graph $(K_{3,3},l)$. The reason for choosing $(K_{3,3},l)$ is that this graph is essentially the simplest graph for which the realizable extension problem is non-trivial. With the exception of $K^{4}$, for which the realizable extension problem is essentially trivial, all graphs smaller than $K_{3,3}$ are series-parallel and so the realizability problem and hence, the realizable extension problem, can be solved using the results of \cite{JML}.\\

Consider the weighted complete bi-partite graph $(K_{3,3}, l)$, where $V_{K_{3,3}}=\{v_{1},v_{2},v_{3},$ $v_{4},v_{5},v_{6}\}$ and $E_{K_{3,3}}=\{v_{6}v_{1},v_{1}v_{2},v_{2}v_{3},v_{3}v_{4},v_{4}v_{5},v_{5}v_{6},v_{1}v_{4}, v_{2}v_{5}$, $v_{3}v_{6}\}$. It is assumed throughout this section that the lengths assigned by $l:E_{K_{3,3}}\to \mathbb{R}^{\geq 0}$ are denoted $ l(v_{1}v_{6})=a$, $ l(v_{1}v_{2})=b$, $ l(v_{2}v_{3})=c$, $ l(v_{3}v_{4})=d$, $ l(v_{4}v_{5})=e$, $ l(v_{5}v_{6})=f$, $ l(v_{1}v_{4})=\alpha$, $ l(v_{3}v_{6})=\beta$ and $ l(v_{2}v_{5})=\gamma$. The values $a,b,...,\gamma$ are not assumed to be fixed at this stage. This notation is illustrated in Fig. \ref{K5}. \\

\begin{figure}[t]
\begin{center}
\scalebox{1}{$\begin{xy}
 \POS (28,12) *\cir<2pt>{} ="a1" ,
 (20,0) *\cir<2pt>{} ="b1" ,
 (8,0) *\cir<2pt>{} ="c1" ,
 (0,12) *\cir<2pt>{} ="d1" ,
 (8,24) *\cir<2pt>{} ="e1" ,
 (20,24) *\cir<2pt>{} ="f1" ,
 (-10,24)*+!{(K_{3,3}, l)},
 (28,12) *+!L{v_{1}},
 (20,0) *+!UL{v_{2}},
 (8,0) *+!UR{v_{3}},
 (0,12)*+!R{v_{4}},
 (8,23) *+!DR{v_{5}},
 (20,23) *+!DL{v_{6}},

\POS "a1" \ar@{-}^{b}  "b1",
\POS "b1" \ar@{-}^{c}  "c1",
\POS "c1" \ar@{-}^{d}  "d1",
\POS "d1" \ar@{-}^{e}  "e1",
\POS "e1" \ar@{-}^{f}  "f1",
\POS "f1" \ar@{-}^{a}  "a1",
\POS "a1" \ar@{-}_>>>>>{\alpha}  "d1",
\POS "b1" \ar@{-}^<<<<<{\gamma}  "e1",
\POS "c1" \ar@{-}^<<<<<{\beta}  "f1",

 \end{xy}$}
\caption{The weighted graph $(K_{3,3}, l)$}
\label{K5}
\end{center}
\end{figure}
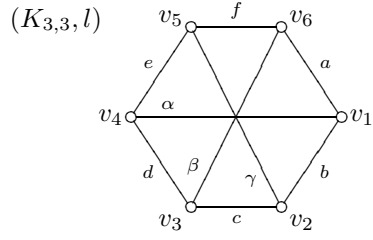

Consider also the four specific subgraphs of $K_{3,3}$ which are defined as  $G_{3}=(V_{K_{3,3}},E_{K_{3,3}}\setminus v_{2}v_{5})$, $G_{2}=(V_{K_{3,3}},E_{G_{3}}\setminus v_{3}v_{6})$, $G_{1}=(V_{K_{3,3}},E_{G_{2}}\setminus v_{1}v_{4})$ and $G_{0}=(V_{K_{3,3}},E_{G_{1}}\setminus v_{5}v_{6})$ which is a path. The former three of the aforementioned subgraphs of $K_{3,3}$ are shown in Fig. \ref{K5b}. \\

\begin{figure}[t]
\begin{center}
\scalebox{1}{$\begin{xy}
 \POS (28,12) *\cir<2pt>{} ="a1" ,
 (20,0) *\cir<2pt>{} ="b1" ,
 (8,0) *\cir<2pt>{} ="c1" ,
 (0,12) *\cir<2pt>{} ="d1" ,
 (8,24) *\cir<2pt>{} ="e1" ,
 (20,24) *\cir<2pt>{} ="f1" ,
 (28,12) *+!L{v_{1}},
 (20,0) *+!UL{v_{2}},
 (8,0) *+!UR{v_{3}},
 (0,12)*+!R{v_{4}},
 (8,23) *+!DR{v_{5}},
 (20,23) *+!DL{v_{6}},
(0,30) *+!{(G_{1}, l_{ G_{1}})},

 \POS "a1" \ar@{-}^{b}  "b1",
\POS "b1" \ar@{-}^{c}  "c1",
\POS "c1" \ar@{-}^{d}  "d1",
\POS "d1" \ar@{-}^{e}  "e1",
\POS "e1" \ar@{-}^{f}  "f1",
\POS "f1" \ar@{-}^{a}  "a1",
\end{xy}$}\scalebox{1}{$\begin{xy}
 \POS (28,12) *\cir<2pt>{} ="a1" ,
 (20,0) *\cir<2pt>{} ="b1" ,
 (8,0) *\cir<2pt>{} ="c1" ,
 (0,12) *\cir<2pt>{} ="d1" ,
 (8,24) *\cir<2pt>{} ="e1" ,
 (20,24) *\cir<2pt>{} ="f1" ,
 (28,12) *+!L{v_{1}},
 (20,0) *+!UL{v_{2}},
 (8,0) *+!UR{v_{3}},
 (0,12)*+!R{v_{4}},
 (8,23) *+!DR{v_{5}},
 (20,23) *+!DL{v_{6}},
(0,30) *+!{(G_{2}, l_{ G_{2}})},

 \POS "a1" \ar@{-}^{b}  "b1",
\POS "b1" \ar@{-}^{c}  "c1",
\POS "c1" \ar@{-}^{d}  "d1",
\POS "d1" \ar@{-}^{e}  "e1",
\POS "e1" \ar@{-}^{f}  "f1",
\POS "f1" \ar@{-}^{a}  "a1",
\POS "a1" \ar@{-}_>>>>>>>>{\alpha}  "d1",

 \end{xy}$}$\begin{xy}
 \POS (28,12) *\cir<2pt>{} ="a1" ,
 (20,0) *\cir<2pt>{} ="b1" ,
 (8,0) *\cir<2pt>{} ="c1" ,
 (0,12) *\cir<2pt>{} ="d1" ,
 (8,24) *\cir<2pt>{} ="e1" ,
 (20,24) *\cir<2pt>{} ="f1" ,
 (28,12) *+!L{v_{1}},
 (20,0) *+!UL{v_{2}},
 (8,0) *+!UR{v_{3}},
 (0,12)*+!R{v_{4}},
 (8,23) *+!DR{v_{5}},
 (20,23) *+!DL{v_{6}},
(0,30) *+!{(G_{3}, l_{ G_{3}})},

 \POS "a1" \ar@{-}^{b}  "b1",
\POS "b1" \ar@{-}^{c}  "c1",
\POS "c1" \ar@{-}^{d}  "d1",
\POS "d1" \ar@{-}^{e}  "e1",
\POS "e1" \ar@{-}^{f}  "f1",
\POS "f1" \ar@{-}^{a}  "a1",
\POS "a1" \ar@{-}_>>>>>>>{\alpha}  "d1",
\POS "f1" \ar@{-}_>>>>>>>{\beta}  "c1",

\end{xy}$

\caption{The weighted graphs $(G_{1}, l_{ G_{1}})$, $(G_{2}, l_{ G_{2}})$ and $(G_{3}, l_{ G_{3}})$ }
\label{K5b}
\end{center}
\end{figure}
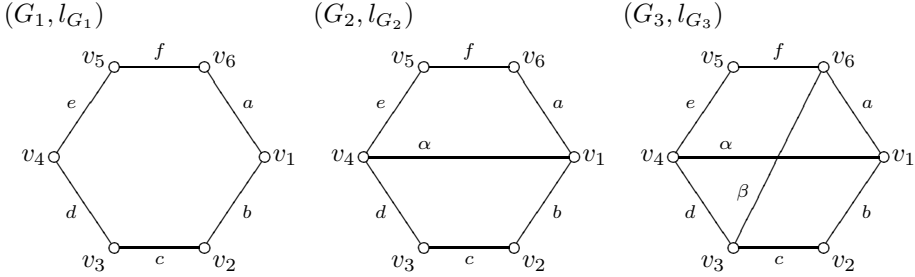

Assuming that $l_{G_{0}}$ is given, thus fixing the edge lengths $a,b,c,d$ and $e$, then determining conditions which the extensions $l_{ G_{1}}, l_{ G_{2}}, l_{ G_{3}}$ and $ l$ must satisfy so that $(G_{1}, l_{ G_{1}})$, $(G_{2}, l_{ G_{2}})$, $(G_{3}, l_{ G_{3}})$ and $(K_{3,3}, l)$, respectively, are realizable, is the focus of the remainder of this section.

\begin{lem}\label{lemmG0}
Given a weighted graph $(G_{0}, l_{ G_{0}})$, as above, then $(G_{1}, l_{ G_{1}})$ is realizable if and only if $ l_{ G_{1}}$ assigns a value for $f$ such that
\begin{center}
 $f\in [\max\{0,2.\max\{a,b,c,d,e\}-(a+b+c+d+e)\},a+b+c+d+e]$ 
\end{center}
\end{lem}

\begin{proof} As $G_{0}$ is a path then $(G_{0}, l_{ G_{0}})$ is always realizable. The graph  $G_{1}=(V_{G_{0}},E_{G_{0}}\cup v_{5}v_{6})$ is a cycle, and so $(G_{1},l_{G_{1}})$ is realizable if and only if the inequality $ f \leq a+b+c+d+e$, and all of the other five cyclic permutations of this inequality, are satisfied. Choosing $f\in [\max\{0,2.\max\{a,b,c,d,e\}-(a+b+c+d+e)\},a+b+c+d+e]$ ensures all six inequalities are satisfied.   
\end{proof}


\begin{lem}
Given a realizable weighted graph $(G_{1}, l_{ G_{1}})$, as above, and letting  $\mu_{1}=2.\max\{a,e,f\}$ and $\mu_{2}=2.\max\{b,c,d\}$, then $(G_{2}, l_{ G_{2}})$ is realizable if and only if $ l_{ G_{2}}$ assigns a value for $\alpha$ such that
\begin{center}
 $\alpha\in [\max\{0,\mu_{1}-(a+e+f)\},a+e+f]\cap [\max\{0,\mu_{2}-(b+c+d)\},b+c+d]$ 
\end{center}

\end{lem}
\begin{proof} Consider the paths $P^{1}$ and $P^{2}$ contained in $G_{1}$ with respective edge sets $E_{P^{1}}=\{v_{4}v_{5},v_{5}v_{6},v_{6}v_{1}\}$ and $E_{P^{2}}=\{v_{1}v_{2},v_{2}v_{3},v_{3}v_{4}\}$. Consider also the cycles $C^{1}$ and $C^{2}$ contained in $G_{2}$ with respective edge sets  $E_{C^{1}}=E_{P^{1}}\cup v_{1}v_{4}$ and $E_{C^{2}}=E_{P^{2}}\cup v_{1}v_{4}$. Clearly $(G_{2},l_{G_{2}})$ is realizable if and only if both $(C^{1},l_{C^{1}})$ and $(C^{2},l_{C^{2}})$ are realizable \textit{and} both $l_{C^{1}}$ and $l_{C^{2}}$ assign the same (permissible) value of $\alpha$ to the edge $v_{1}v_{4}$. It now follows from Lemma \ref{lemmG0} that $\alpha\in [\max\{0,\mu_{1}-(a+e+f)\},a+e+f]\cap [\max\{0,\mu_{2}-(b+c+d)\},b+c+d]$ where $\mu_{1}=2.\max\{a,e,f\}$ and $\mu_{2}=2.\max\{b,c,d\}$.  
\end{proof}

Before considering the weighted graph $(G_{3}, l_{ G_{3}})$ the concept of a \textit{workspace} is introduced. For more details regarding workspaces see \cite{CS}, \cite{MT} or \cite{TW}, where the concept first appears. Given a weighted graph $(G,l)$, then the \textit{workspace} of a vertex $v$ with respect to the graph $G$, the length function $l$ and an edge $ab\in E_{G}$ where $l(ab)>0$, is defined as the image of the map $M(G,l)\to M(H,l_{\mid H})$ i.e.
\begin{center}$W_{G,l,ab}(v)=im(M(G,l)\to M(H,l_{\mid H}))$ \end{center}
where $H=(\{a,b,v\},\{ab\})$ and $l_{\mid H}$ is the restriction of $l$ induced by $H\subset G$.\\

Note that the moduli space $M(H,l_{\mid H})$ is in fact a copy of $\mathbb E^2$. It is possible to construct an explicit homeomorphism $\varphi_{a}$ as follows. For each $[p] \in M(H,l_{\mid H})$, let $q$ be the unique realization in $C(H,l_{\mid H})$ that satisfies $q(a) = (0,0)$, $q(b) = (l(ab),0)$, and $[q]=[p]$ in $M(H,l_{\mid H})$. It is now possible to define $\varphi_{a} ([p]) = q(v)$. It is clear that $\varphi_{a}: M(H,l_{\mid H}) \rightarrow \mathbb{E}^2$ is a homeomorphism.  In the sequel, the map $\varphi_{a}$ is used to identify the workspace of a vertex with a particular subset of $\mathbb{E}^2$. 

\begin{lem}\label{varphi}
Given a realizable weighted graph $(G_{2}, l_{ G_{2}})$, as above, then the subset of $\mathbb{R}^{\geq 0}$ from which the value of $\beta=l_{ G_{3}}(v_{3}v_{6})$ can be chosen so that $(G_{3}, l_{ G_{3}})$ is realizable is an interval or the disjoint union of two intervals.
\end{lem}

\begin{proof}  Given a weighted graph $(C,l)$ where $C$ is a cycle such that $ij,jk\in E_{C}$, then it is well known, see \cite{CS}, that the image of $\varphi_{i}|_{ W_{C,l,ij}(k)}$ has one of three types; a circle $S$ with centre $(l(ij),0)$ and radius $l(jk)$, a contractible subset of $S$ or two disjoint contractible subsets of $S$. All three of these subsets of $S$ are also symmetric about the $x$-axis i.e. $w\in im(\varphi_{i}|_{ W_{C,l,ij}(k)})\iff \rho_{x}(w)\in im(\varphi_{i}|_{ W_{C,l,ij}(k)})$ where $\rho_{x}$ is the reflection in the $x$-axis. Returning to the $(G_{3},l_{G_{3}})$ case at hand. Consider the circle $S_{1}$ with centre $(l_{G_{2}}(v_{1}v_{4}),0)$ and radius $l_{G_{2}}(v_{1}v_{6})$ and the circle $S_{2}$ with centre $(0,0)$ and radius $l_{G_{2}}(v_{3}v_{4})$. Observe that the images of $\varphi_{v_{4}}|_{W_{G_{2},l_{G_{2}},v_{4}v_{1}}(v_{3})}$ and $\varphi_{v_{4}}|_{W_{G_{2},l_{G_{2}},v_{4}v_{1}}(v_{6})}$ are subsets of circles $S_{1}$ and $S_{2}$, respectively, and these images are denoted $W(v_{3})$ and $W(v_{6})$, respectively, for the rest of this proof. The structure of the set $X=\{d(w,w')\mid w\in W(v_{3})\; \textrm{and }\; w'\in  W(v_{6}) \}$ is now determined. \\

Consider the value $m=\min\{d(w,w')\mid w\in W(v_{3}) \;\textrm{and }\; w'\in W(v_{6})\}$ and the value $N=\max\{d(w,w')\mid w\in W(v_{3}) \;\textrm{and }\; w'\in W(v_{6})\}$. A brief consideration of subsets of two circles (centred on the $x$-axis) which are symmetric about the $x$-axis leads to the conclusion that there is only one case where $X\neq [m,N]$. This case is a special case of the instance where $W(v_{3})$ and $W(v_{6})$ are themselves two disjoint contractible subsets of $S_{1}$ and $S_{2}$ respectively. In order to describe this special case denote by $W(v_{3})^{+}$ the component of $W(v_{3})$ contained in the upper half-plane and denote by $W(v_{3})^{-}$ the component of $W(v_{3})$ contained in the lower half-plane. The components $W(v_{6})^{+}$ and $W(v_{6})^{-}$ of $W(v_{6})$ are defined similarly. Now, consider the value $M=\max\{d(w,w')\mid w\in W(v_{3})^{+} \;\textrm{and }\; w'\in W(v_{6})^{+}\}$ and the value $n=\min\{d(w,w')\mid w\in W(v_{3})^{-} \;\textrm{and }\; w'\in W(v_{6})^{+}\}$. The aforementioned special case occurs whenever $n>M$ and so the subset of $\mathbb{R}^{\geq 0}$ from which the value of $\beta$ can be chosen so that $(G_{3}, l_{ G_{3}})$ is realizable is the disjoint union of two intervals $[m,M]\sqcup[n,N]$. Consider Fig. \ref{disjoint} and note that the subsets $W(v_{3})=W(v_{3})^{+}\sqcup W(v_{3})^{-}=[w_{1},w_{2}]\sqcup[w_{3},w_{4}]$ and $W(v_{6})=W(v_{6})^{+}\sqcup W(v_{6})^{-}=[w_{5},w_{6}]\sqcup[w_{7},w_{8}]$ of the circles $S_{1}$ and $S_{2}$, respectively, and let $L=l_{G_{2}}(v_{1}v_{4})$. \\

\begin{figure}[t]
\begin{center}
\scalebox{0.95}{$\begin{xy}
\POS (-6,12) *\cir<2pt>{} ="a" *+!U{^{ (0,0)}} ,
 (30,12) *\cir<2pt>{} ="b" *+!U{^{(L,0)}} ,
 (30,26) *\cir<2pt>{} ="p1" *+!L{^{w_{1}}},
 (26,25) *\cir<2pt>{} ="p2" ,
 (27.3,21.3) *+!{^{w_{2}}},
 (26,-1) *\cir<2pt>{} ="p3" *+!U{^{w_{3}}},
 (30,-2) *\cir<2pt>{} ="p4" *+!U{^{w_{4}}},
 (0,24) *\cir<2pt>{} ="q1" *+!R{^{w_{5}}} ,
 (3,22) *\cir<2pt>{} ="q2" *+!U{^{w_{6}}} ,
 (3,2) *\cir<2pt>{} ="q3" *+!D{^{w_{7}}},
 (0,0) *\cir<2pt>{} ="q4" *+!R{^{w_{8}}},
 (2.5,-2.5) *\cir<0pt>{} ="qq4" ,
 (31,24.3) *\cir<0pt>{} ="pp1" ,
(0,26) *\cir<0pt>{} ="qq1" ,
 (29.5,29) *\cir<0pt>{} ="pp2" ,
 (44,12) *\cir<0pt>{} ="v" ,
  (3,2)*\cir<2pt>{} ="e" ,
 (-9,26)*+!{^{S_{2}}},
 (39,26)*+!{^{S_{1}}},
 (-30,12) *\cir<0pt>{} ="a1" ,
 (55,12) *\cir<0pt>{} ="b1",
  
\POS "a1" \ar@{-}  "b1", 
\POS "p1" \ar@/_.1pc/@{-}  "p2",
\POS "p3" \ar@/_.1pc/@{-}  "p4",
\POS "q1" \ar@/^.1pc/@{-}  "q2",
\POS "q3" \ar@/^.1pc/@{-}  "q4",
\POS "pp1" \ar@{|<-->|}^{N}  "qq4",
\POS "p2" \ar@{<-->}_{n}  "q3",
\POS "pp2" \ar@{|<-->|}_{M}  "qq1",
\POS "q2" \ar@{<-->}_{m}  "p2",

\POS (-6,12){\ellipse(13.5){.}}
\POS (30,12){\ellipse(14){.}}
 \end{xy}$}
\caption{The subset of $\mathbb{R}^{\geq 0}$ from which the value of $\beta$ can be chosen so that $(G_{3}, l_{ G_{3}})$ is realizable can be the disjoint union $[m,M]\sqcup [n,N]$}
\label{disjoint}
\end{center}
\end{figure}
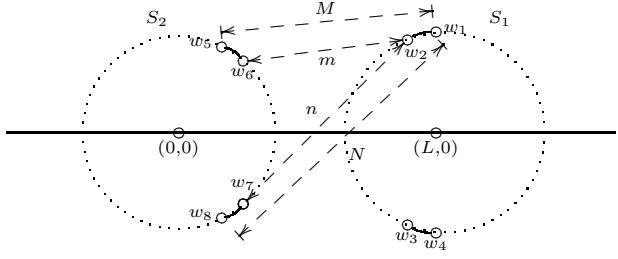

Hence, the subset of $\mathbb{R}^{\geq 0}$ from which the value of $\beta$ can be chosen so that $(G_{3}, l_{ G_{3}})$ is realizable is either an interval $[m,N]$ or the disjoint union of two intervals $[m,M]\sqcup [n,N]$, where $m,M,n$ and $N$ are defined as above. 
\end{proof}

\begin{lem}
Given a realizable weighted graph $(G_{3}, l_{ G_{3}})$, as above, then the subset of $\mathbb{R}^{\geq 0}$ from which the value of $\gamma=l(v_{2}v_{5})$ can be chosen so that $(K_{3,3}, l)$ is realizable is an interval or the disjoint union of two, three or four intervals.
\end{lem}

\begin{proof} 
Consider a weighted graph $(H, h)$, where $V_{H}=\{u_{1},u_{2},u_{3},u_{4},u_{5}\}$ and $E_{H}=\{u_{1}u_{2},u_{1}u_{3},u_{1}u_{4},u_{1}u_{5},u_{2}u_{3},u_{3}u_{4},u_{4}u_{5}\}$ as shown in Fig. \ref{K16b}. Observe that $M(H, h)$ is homeomorphic to $C_{u_{4},u_{1}}(H,h)$. If $(H, h)$ is realizable, then for every $q\in C_{u_{4},u_{1}}(H, h)$ there exists a $\rho q\in C_{u_{4},u_{1}}(H, h)$ whose image is a reflection of the image of $q$ in the $x$-axis. Observe that $C_{u_{4},u_{1}}(H, h)$ can have at most $2^{3}$ connected components. Further motivation of this statement is provided in Fig. \ref{K16b}. The images of the realizations $p,q,r$ and $s$ of $(H,h)$ are shown, and there also exists four corresponding realizations $\rho p,\rho q,\rho r$ and $\rho s$ of $(H,h)$ in $C_{u_{4},u_{1}}(H, h)$ which are not shown. Clearly it is possible to chose length functions $ l_{ G_{3}}$ and $ h$ so that $b:=l_{ G_{3}}(v_{1}v_{2})= h(u_{1}u_{2}),c:=l_{ G_{3}}(v_{2}v_{3})= h(u_{2}u_{3}),d:=l_{ G_{3}}(v_{3}v_{4})= h(u_{3}u_{4}),e:=l_{ G_{3}}(v_{4}v_{5})= h(u_{4}u_{5}),f:=l_{ G_{3}}(v_{5}v_{1})= h(u_{5}u_{1}),\alpha:=l_{ G_{3}}(v_{1}v_{4})= h(u_{1}u_{4})$ and $\beta:=l_{ G_{3}}(v_{1}v_{3})= h(u_{1}u_{3})$ and that the edge length $a:=l_{ G_{3}}(v_{1}v_{6})$ is assigned an arbitrarily small length $\epsilon\ll 1$ by $ l_{ G_{3}}$. Consider the image of the realization $p\in C_{u_{4},u_{1}}(G_{3}, l_{ G_{3}})$ which is contained in Fig. \ref{K13a}. Note that the  incidence structure of the larger node, labeled $^{p_{6}}_{p_{1}}$, is shown in the detailed (\textit{blown-up}) section contained in the circle on the right-hand-side of Fig. \ref{K13a}. \\

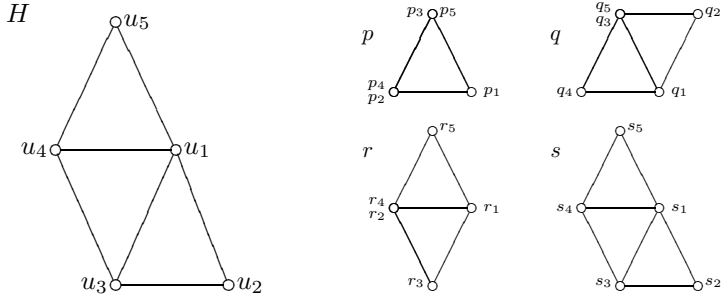
\begin{figure}[t]
\begin{center}
$\begin{xy}
\POS (13,18) *\cir<2pt>{} ="a" *+!L{u_{1}} ,
 (20,0) *\cir<2pt>{} ="b" *+!L{u_{2}},
 (5,0) *\cir<2pt>{} ="c" *+!R{u_{3}},
 (-3,18) *\cir<2pt>{} ="d" *+!R{u_{4}},
 (5,35) *\cir<2pt>{} ="e" *+!L{u_{5}},
  (-8,35)*+!{H},

\POS "a" \ar@{-}  "b",
\POS "a" \ar@{-}  "c",
\POS "a" \ar@{-}  "d",
\POS "a" \ar@{-}  "e",

\POS "b" \ar@{-}  "c",
\POS "c" \ar@{-}  "d",
\POS "d" \ar@{-}  "e",

\end{xy}$\hspace{1cm} \scalebox{0.86}{$\begin{xy}
\POS (12,30) *\cir<2pt>{} ="a" *+!L{_{\; p_{1}}} ,
 (0,30) *\cir<2pt>{} ="b" *+!R{_{ p_{2}}^{p_{4}}},
 (6,42) *\cir<2pt>{} ="c" *+!R{^{ p_{3}}},
 (0,30) *\cir<2pt>{} ="d" ,
 (6,42) *\cir<2pt>{} ="e" *+!L{^{ p_{5}}},
  (-4,38)*+!{p},
 
\POS "a" \ar@{-}  "b",
\POS "a" \ar@{-}  "c",
\POS "a" \ar@{-}  "d",
\POS "a" \ar@{-}  "e",
\POS "b" \ar@{-}  "c",
\POS "c" \ar@{-}  "d",
\POS "d" \ar@{-}  "e", 

\POS (41,30) *\cir<2pt>{} ="a" *+!L{_{\; q_{1}}} ,
 (47,42) *\cir<2pt>{} ="b" *+!L{^{ q_{2}}},
 (35,42) *\cir<2pt>{} ="c" *+!R{_{ q_{3}}^{q_{5}}},
 (29,30) *\cir<2pt>{} ="d" *+!R{_{ q_{4}}},
 (35,42) *\cir<2pt>{} ="e" ,
  (25,38)*+!{q},
 
\POS "a" \ar@{-}  "b",
\POS "a" \ar@{-}  "c",
\POS "a" \ar@{-}  "d",
\POS "a" \ar@{-}  "e",
\POS "b" \ar@{-}  "c",
\POS "c" \ar@{-}  "d",
\POS "d" \ar@{-}  "e", 

\POS (12,12) *\cir<2pt>{} ="a" *+!L{_{\; r_{1}}} ,
 (0,12) *\cir<2pt>{} ="b" *+!R{_{ r_{2}}^{r_{4}}},
 (6,0) *\cir<2pt>{} ="c" *+!R{^{ r_{3}}},
 (0,12) *\cir<2pt>{} ="d" ,
 (6,24) *\cir<2pt>{} ="e" *+!L{^{ r_{5}}},
  (-4,20)*+!{r},
 
\POS "a" \ar@{-}  "b",
\POS "a" \ar@{-}  "c",
\POS "a" \ar@{-}  "d",
\POS "a" \ar@{-}  "e",
\POS "b" \ar@{-}  "c",
\POS "c" \ar@{-}  "d",
\POS "d" \ar@{-}  "e", 

\POS (41,12) *\cir<2pt>{} ="a" *+!L{_{\; s_{1}}} ,
 (47,0) *\cir<2pt>{} ="b" *+!L{^{ s_{2}}},
 (35,0) *\cir<2pt>{} ="c" *+!R{^{ s_{3}}},
 (29,12) *\cir<2pt>{} ="d" *+!R{_{ s_{4}}},
 (35,24) *\cir<2pt>{} ="e" *+!L{^{ s_{5}}},
  (25,20)*+!{s},
 
\POS "a" \ar@{-}  "b",
\POS "a" \ar@{-}  "c",
\POS "a" \ar@{-}  "d",
\POS "a" \ar@{-}  "e",
\POS "b" \ar@{-}  "c",
\POS "c" \ar@{-}  "d",
\POS "d" \ar@{-}  "e", 
\end{xy}$}

\caption{The graph $H$ and the images of the realizations $p$, $q$, $r$ and $s$ of $(H,h)$ given an equilateral length function $h$}
\label{K16b}
\end{center}
\end{figure}

\begin{figure}[b]
\begin{center}
\scalebox{1}{$\begin{xy}
\POS (20,20) *\cir<4pt>{} ="a"  ,
(20.5,18.5)  *+!U{_{p_{1}}} ,
(20.5,20)  *+!D{^{p_{6}}} ,
 (32,0) *\cir<2pt>{} ="b" *+!L{^{ p_{2}}},
 (10,0) *\cir<2pt>{} ="c" *+!R{^{ p_{3}}},
 (0,20) *\cir<2pt>{} ="d" *+!R{^{ p_{4}}},
 (10,38) *\cir<2pt>{} ="e" *+!L{^{ p_{5}}},
 (55,20) *\cir<2pt>{} ="h" *+!U{_{ p_{1}}},
 (55,25) *\cir<2pt>{} ="i" *+!D{^{ p_{6}}},
 (50,33) *\cir<0pt>{} ="j" *+!D{^{\rm{to}\; p_{5}}},
 (50,7) *\cir<0pt>{} ="j1",
 (41,20) *\cir<0pt>{} ="k" *+!R{^{\rm{to}\; p_{4}}},
 (45,10) *\cir<0pt>{} ="l" *+!UR{^{\rm{to}\; p_{3}}},
 (65,10) *\cir<0pt>{} ="m" *+!UL{^{\rm{to}\; p_{2}}},

\POS "a" \ar@{-}^{b}  "b",
\POS "a" \ar@{-}_{\beta}  "c",
\POS "a" \ar@{-}_{\alpha}  "d",
\POS "a" \ar@{-}_{f}  "e",
\POS "a" \ar@{.}  "j",
\POS "a" \ar@{.}  "j1",
\POS "b" \ar@{-}^{c}  "c",
\POS "c" \ar@{-}^{d}  "d",
\POS "d" \ar@{-}^{e}  "e",
\POS "h" \ar@{-}_{a=\epsilon}  "i",
\POS "h" \ar@{-}  "k",
\POS "h" \ar@{-}  "m",
\POS "i" \ar@{-}  "j",
\POS "i" \ar@{-}  "l",

\POS (55,20){\ellipse(14){}}

\end{xy}$}

\caption{The image of a realization $p$ of the weighted graph $(G_{3}, l_{ G_{3}})$ where $l_{G_{3}}$ assigns the edge $v_{1}v_{6}$ the length $\epsilon$ such that $l_{ G_{3}}(v_{1}v_{6})= a=\epsilon \ll 1$}
\label{K13a}
\end{center}
\end{figure}
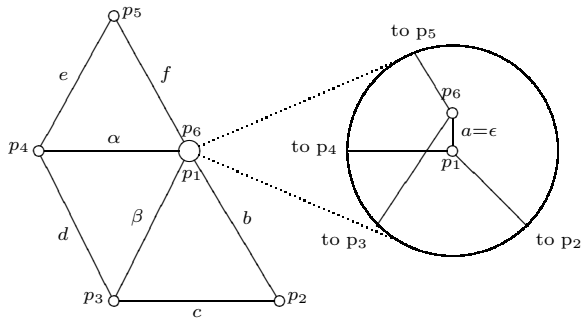

As a result of such length functions, then each connected component of the moduli space $M(G_{3}, l_{ G_{3}})$ is a circle whereas each connected component of the moduli space $M(H, h)$ is a point. The salient point here is that the connected components of $M(H, h)$ and $M(G_{3}, l_{ G_{3}})$ are in a one-to-one correspondence.\\
 
Observe that if $C_{u_{4},u_{1}}(G_{3}, l_{ G_{3}})$ has $8$ connected components then these eight components must occur in pairs such that for any realization $p$ contained in one component of  $C_{u_{4},u_{1}}(G_{3}, l_{ G_{3}})$ there exists a realization, denoted $\rho p$, in another component of $C_{u_{4},u_{1}}(G_{3}, l_{ G_{3}})$ such that the image of $\rho p$ is a reflection of the image of $p$ in the $x$-axis. As a reflection is an isometry, then the distance $d(p_{2}, p_{5})$ must be equal to the distance $ d(\rho p(v_{2}), \rho p(v_{5}))$, for each $p\in C_{u_{4},u_{1}}(G_{3}, l_{ G_{3}})$. Hence, the subset of $\mathbb{R}^{\geq 0}$ from which the value of $\gamma=l(v_{2}v_{5})$ can be chosen so that $(K_{3,3}, l)$ is realizable, can have at most four connected components. Examples \ref{ex1}, \ref{ex2}, \ref{ex3} and \ref{ex4} are, respectively, occurrences of the subset of $\mathbb{R}^{\geq 0}$ from which the value of $\gamma$ can be chosen so that $(K_{3,3}, l)$ is realizable being one, two, three and four, disjoint intervals. This completes the proof.
\end{proof}


\begin{ex}\label{ex1} Suppose that $a=b=c=d=e=f=\alpha=\beta =1$ then the subset of $\mathbb{R}^{\geq 0}$ from which the value of $\gamma$ can be chosen so that $(K_{3,3}, l)$ is realizable is the set $\{1\}$. See Figure \ref{K13e}.

\begin{figure}[H]
\begin{center}
{
$\begin{xy}
\POS (12,0) *\cir<2pt>{} ="a" *+!L{^{ p_{1}}} ,
 (12,12) *\cir<2pt>{} ="b" *+!L{_{ p_{2}}^{ p_{6}}},
 (0,12) *\cir<2pt>{} ="c" *+!R{_{ p_{3}}^{ p_{5}}},
 (0,0) *\cir<2pt>{} ="d" *+!R{^{ p_{4}}},
 (0,12) *\cir<2pt>{} ="e" ,
  (-10,12)*+!{p},
 
\POS "a" \ar@{-}  "b",
\POS "a" \ar@{-}  "d",
\POS "b" \ar@{-}  "c",
\POS "c" \ar@{-}  "d",
\POS "d" \ar@{-}  "e", 
\end{xy}$
}

\caption{The subset of $\mathbb{R}^{\geq 0}$ from which the value of $\gamma$ can be chosen so that $(K_{3,3}, l)$ is realizable is a single point}
\label{K13e}
\end{center}
\end{figure}
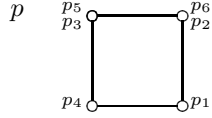
\end{ex}

In Examples \ref{ex2}, \ref{ex3} and \ref{ex4} the larger nodes labeled $^{\; p_{6}}_{\; p_{1}}$, $^{\; q_{6}}_{\; q_{1}}$, $^{\; r_{6}}_{\; r_{1}}$ and $^{\; s_{6}}_{\; s_{1}}$ are each analogous to the larger node contained in Fig. \ref{K13a} i.e. they possess the same incidence structure. It should also be noted that, in the interest of brevity, the images of the four realizations $\rho p, \rho q, \rho r$ and $\rho s$ of $(K_{3,3},l)$, (whose images are the reflections in the $x$-axis of the images of $p$, $q$, $r$ and $s$, respectively) are omitted from Fig. \ref{K13d}, Fig. \ref{K13c} and Fig. \ref{K13b}.

\begin{ex}\label{ex2} Suppose that $a=\epsilon \ll b=c=d=e=\alpha=1$ and $f=\beta=\sqrt{2}$ then the subset of $\mathbb{R}^{\geq 0}$ from which the value of $\gamma=l(v_{2}v_{5})$ can be chosen so that $(K_{3,3}, l)$ is realizable is $[1-\delta_{1},1+\delta_{1}]$ $\sqcup$ $[\sqrt{5}-\delta_{2},\sqrt{5}-\delta_{2}]$ where $\delta_{1},\delta_{2} \ll 1$. See Figure \ref{K13d}.

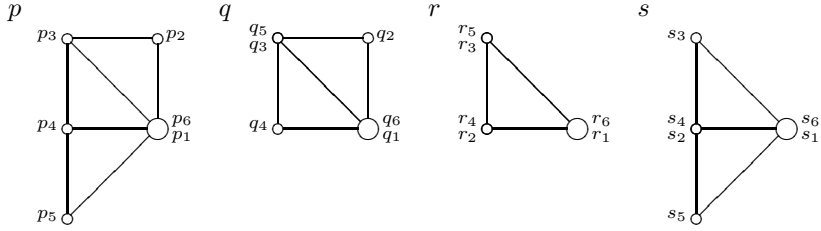
\begin{figure}[t]
\begin{center}
{$\begin{xy}
\POS (12,0) *\cir<4pt>{} ="a" *+!L{^{\; p_{6}}_{\; p_{1}}} ,
 (12,12) *\cir<2pt>{} ="b" *+!L{^{ p_{2}}},
 (0,12) *\cir<2pt>{} ="c" *+!R{^{ p_{3}}},
 (0,0) *\cir<2pt>{} ="d" *+!R{^{ p_{4}}},
 (0,-12) *\cir<2pt>{} ="e" *+!R{^{ p_{5}}},
  (-7,15)*+!{p},
 
\POS "a" \ar@{-}  "b",
\POS "a" \ar@{-}  "c",
\POS "a" \ar@{-}  "d",
\POS "a" \ar@{-}  "e",
\POS "b" \ar@{-}  "c",
\POS "c" \ar@{-}  "d",
\POS "d" \ar@{-}  "e", 
\end{xy}$
$\begin{xy}
\POS (12,0) *\cir<4pt>{} ="a" *+!L{^{\; q_{6}}_{\; q_{1}}} ,
 (12,12) *\cir<2pt>{} ="b" *+!L{^{ q_{2}}},
 (0,12) *\cir<2pt>{} ="c" *+!R{_{ q_{3}}^{ q_{5}}},
 (0,0) *\cir<2pt>{} ="d" *+!R{^{ q_{4}}},
 (0,12) *\cir<2pt>{} ="e" ,
  (-7,15)*+!{q},
 
\POS "a" \ar@{-}  "b",
\POS "a" \ar@{-}  "c",
\POS "a" \ar@{-}  "d",
\POS "a" \ar@{-}  "e",
\POS "b" \ar@{-}  "c",
\POS "c" \ar@{-}  "d",
\POS "d" \ar@{-}  "e", 
\end{xy}$
$\begin{xy}
\POS (12,0) *\cir<4pt>{} ="a" *+!L{^{\; r_{6}}_{\; r_{1}}} ,
 (0,0) *\cir<2pt>{} ="b" *+!R{_{ r_{2}}^{ r_{4}}},
 (0,12) *\cir<2pt>{} ="c" *+!R{_{ r_{3}}^{ r_{5}}},
 (0,0) *\cir<2pt>{} ="d",
 (0,12) *\cir<2pt>{} ="e" ,
  (-7,15)*+!{r},
 
\POS "a" \ar@{-}  "b",
\POS "a" \ar@{-}  "c",
\POS "a" \ar@{-}  "d",
\POS "a" \ar@{-}  "e",
\POS "b" \ar@{-}  "c",
\POS "c" \ar@{-}  "d",
\POS "d" \ar@{-}  "e", 
\end{xy}$
$\begin{xy}
\POS (12,0) *\cir<4pt>{} ="a" *+!L{^{\; s_{6}}_{\; s_{1}}} ,
 (0,0) *\cir<2pt>{} ="b" *+!R{_{ s_{2}}^{ s_{4}}},
 (0,12) *\cir<2pt>{} ="c" *+!R{^{ s_{3}}},
 (0,0) *\cir<2pt>{} ="d" ,
  (0,-12) *\cir<2pt>{} ="e" *+!R{^{ s_{5}}},
  (-7,15)*+!{s},
 
\POS "a" \ar@{-}  "b",
\POS "a" \ar@{-}  "c",
\POS "a" \ar@{-}  "d",
\POS "a" \ar@{-}  "e",
\POS "b" \ar@{-}  "c",
\POS "c" \ar@{-}  "d",
\POS "d" \ar@{-}  "e", 
\end{xy}$}

\caption{The subset of $\mathbb{R}^{\geq 0}$ from which the value of $\gamma$ can be chosen so that $(K_{3,3}, l)$ is realizable is two disjoint intervals}
\label{K13d}
\end{center}
\end{figure}
\end{ex}


\begin{ex}\label{ex3} Suppose that $a=\epsilon \ll b=c=\frac{\sqrt{5}}{2},d=\alpha=1,e=f=\frac{\sqrt{5}}{4}$ and $\beta=\sqrt{2}$ then the subset of $\mathbb{R}^{\geq 0}$ from which the value of $\gamma=l(v_{2}v_{5})$ can be chosen so that $(K_{3,3}, l)$ is realizable is $[\frac{\sqrt{5}}{4}-\delta_{1},\frac{\sqrt{5}}{4}+\delta_{1}]$ $\sqcup$ $[\frac{\sqrt{13}}{2}-\delta_{2},\frac{\sqrt{13}}{2}+\delta_{2}]$ $\sqcup$ $[\frac{\sqrt{29}}{2}-\delta_{3},\frac{\sqrt{29}}{2}+\delta_{3}]$ where $\delta_{1},\delta_{2},\delta_{3}\ll 1$. See Figure \ref{K13c}.

\begin{figure}[t]
\begin{center}
{$\begin{xy}
\POS (12,0) *\cir<4pt>{} ="a" *+!L{^{\; p_{6}}_{\; p_{1}}} ,
 (18,18) *\cir<2pt>{} ="b" *+!L{^{ p_{2}}},
 (0,12) *\cir<2pt>{} ="c" *+!R{^{ p_{3}}},
 (0,0) *\cir<2pt>{} ="d" *+!R{^{ p_{4}}},
 (6,12) *\cir<2pt>{} ="e" *+!L{^{ p_{5}}},
  (-2,18)*+!{p},
 
\POS "a" \ar@{-}  "b",
\POS "a" \ar@{-}  "c",
\POS "a" \ar@{-}  "d",
\POS "a" \ar@{-}  "e",
\POS "b" \ar@{-}  "c",
\POS "c" \ar@{-}  "d",
\POS "d" \ar@{-}  "e", 
\end{xy}$
$\begin{xy}
\POS (18,0) *\cir<4pt>{} ="a" *+!L{^{\; q_{6}}_{\; q_{1}}} ,
 (0,-6) *\cir<2pt>{} ="b" *+!U{^{ q_{2}}},
 (6,12) *\cir<2pt>{} ="c" *+!R{^{ q_{3}}},
 (6,0) *\cir<2pt>{} ="d" *+!UR{^{ q_{4}}},
 (12,12) *\cir<2pt>{} ="e" *+!L{^{ q_{5}}},
  (-2,18)*+!{q},
 
\POS "a" \ar@{-}  "b",
\POS "a" \ar@{-}  "c",
\POS "a" \ar@{-}  "d",
\POS "a" \ar@{-}  "e",
\POS "b" \ar@{-}  "c",
\POS "c" \ar@{-}  "d",
\POS "d" \ar@{-}  "e", 
\end{xy}$
$\begin{xy}
\POS (18,0) *\cir<4pt>{} ="a" *+!L{^{\; r_{6}}_{\; r_{1}}} ,
 (0,-6) *\cir<2pt>{} ="b" *+!U{^{ r_{2}}},
 (6,12) *\cir<2pt>{} ="c" *+!R{^{ r_{3}}},
 (6,0) *\cir<2pt>{} ="d" *+!UR{^{ r_{4}}},
 (12,-12) *\cir<2pt>{} ="e" *+!L{^{ r_{5}}},
  (-2,18)*+!{r},
 
\POS "a" \ar@{-}  "b",
\POS "a" \ar@{-}  "c",
\POS "a" \ar@{-}  "d",
\POS "a" \ar@{-}  "e",
\POS "b" \ar@{-}  "c",
\POS "c" \ar@{-}  "d",
\POS "d" \ar@{-}  "e", 
\end{xy}$
$\begin{xy}
\POS (12,0) *\cir<4pt>{} ="a" *+!L{^{\; s_{6}}_{\; s_{1}}} ,
 (18,18) *\cir<2pt>{} ="b" *+!L{^{ s_{2}}},
 (0,12) *\cir<2pt>{} ="c" *+!R{^{ s_{3}}},
 (0,0) *\cir<2pt>{} ="d" *+!R{^{ s_{4}}},
  (6,-12) *\cir<2pt>{} ="e" *+!L{^{ s_{5}}},
  (-2,18)*+!{s},
 
\POS "a" \ar@{-}  "b",
\POS "a" \ar@{-}  "c",
\POS "a" \ar@{-}  "d",
\POS "a" \ar@{-}  "e",
\POS "b" \ar@{-}  "c",
\POS "c" \ar@{-}  "d",
\POS "d" \ar@{-}  "e", 
\end{xy}$}

\caption{The subset of $\mathbb{R}^{\geq 0}$ from which the value of $\gamma$ can be chosen so that $(K_{3,3}, l)$ is realizable is three disjoint intervals}
\label{K13c}
\end{center}
\end{figure}
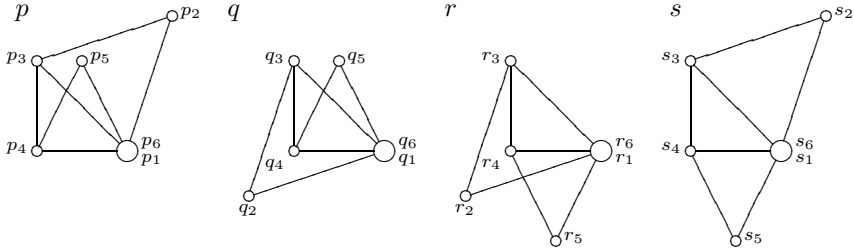
\end{ex}

\begin{ex}\label{ex4} Suppose that $a=\epsilon \ll b=\sqrt{17}, c=\sqrt{29},d=\sqrt{5}, \alpha=3, e=\sqrt{13}, f=\sqrt{10}$ and $\beta=5\sqrt{2}$ then the subset of $\mathbb{R}^{\geq 0}$ from which the value of $\gamma=l(v_{2}v_{5})$ can be chosen so that $(K_{3,3}, l)$ is realizable is  $[\ell_{1}-\delta_{1},\ell_{1}+\delta_{1}]$ $\sqcup$ $[\ell_{2}-\delta_{2},\ell_{2}+\delta_{2}]$ $\sqcup$  $[\ell_{3}-\delta_{3},\ell_{3}+\delta_{3}]$ $\sqcup$ $[\ell_{4}-\delta_{4},\ell_{4}+\delta_{4}]$ where $\delta_{1},\delta_{2},\delta_{3},\delta_{4}\ll 1$ and $\ell_{1},\ell_{2},\ell_{3},\ell_{4}$ are all distinct and $d(\ell_{i},\ell_{j})>\delta_{i}+\delta_{j}$ for all distinct $2$-element subsets $\{i,j\}$ contained in $\{1,2,3,4\}$. See Figure \ref{K13b}.

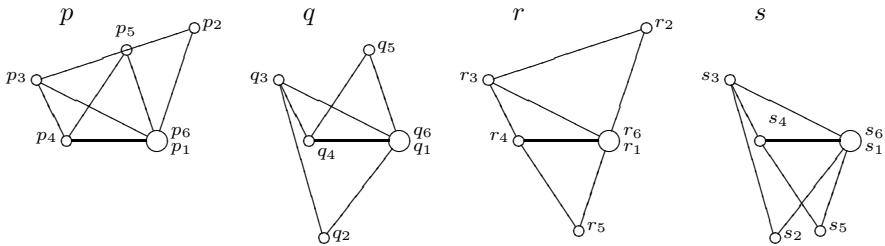
\begin{figure}[b]
\begin{center}
{$\begin{xy}
\POS (16,0) *\cir<4pt>{} ="a" *+!L{^{\; p_{6}}_{\; p_{1}}} ,
 (21,15) *\cir<2pt>{} ="b" *+!L{^{ p_{2}}},
 (0,8) *\cir<2pt>{} ="c" *+!R{^{ p_{3}}},
 (4,0) *\cir<2pt>{} ="d" *+!R{^{ p_{4}}},
 (12,12) *\cir<2pt>{} ="e" *+!D{^{ p_{5}}},
  (4,16)*+!{p},
 
\POS "a" \ar@{-}  "b",
\POS "a" \ar@{-}  "c",
\POS "a" \ar@{-}  "d",
\POS "a" \ar@{-}  "e",
\POS "b" \ar@{-}  "c",
\POS "c" \ar@{-}  "d",
\POS "d" \ar@{-}  "e", 
\end{xy}$
$\begin{xy}
\POS (16,0) *\cir<4pt>{} ="a" *+!L{^{\; q_{6}}_{\; q_{1}}} ,
 (6,-13) *\cir<2pt>{} ="b" *+!L{^{ q_{2}}},
 (0,8) *\cir<2pt>{} ="c" *+!R{^{ q_{3}}},
 (4,0) *\cir<2pt>{} ="d" *+!UL{^{ q_{4}}},
  (12,12) *\cir<2pt>{} ="e" *+!L{^{ q_{5}}},
  (4,16)*+!{q},
 
\POS "a" \ar@{-}  "b",
\POS "a" \ar@{-}  "c",
\POS "a" \ar@{-}  "d",
\POS "a" \ar@{-}  "e",
\POS "b" \ar@{-}  "c",
\POS "c" \ar@{-}  "d",
\POS "d" \ar@{-}  "e", 
\end{xy}$
$\begin{xy}
\POS(16,0) *\cir<4pt>{} ="a" *+!L{^{\; r_{6}}_{\; r_{1}}} ,
 (21,15) *\cir<2pt>{} ="b" *+!L{^{ r_{2}}},
 (0,8) *\cir<2pt>{} ="c" *+!R{^{ r_{3}}},
 (4,0) *\cir<2pt>{} ="d" *+!R{^{ r_{4}}},
  (12,-12) *\cir<2pt>{} ="e" *+!L{^{ r_{5}}},
  (4,16)*+!{r},
 
\POS "a" \ar@{-}  "b",
\POS "a" \ar@{-}  "c",
\POS "a" \ar@{-}  "d",
\POS "a" \ar@{-}  "e",
\POS "b" \ar@{-}  "c",
\POS "c" \ar@{-}  "d",
\POS "d" \ar@{-}  "e", 
\end{xy}$
$\begin{xy}
\POS (16,0) *\cir<4pt>{} ="a" *+!L{^{\; s_{6}}_{\; s_{1}}} ,
 (6,-13) *\cir<2pt>{} ="b" *+!L{^{ s_{2}}},
 (0,8) *\cir<2pt>{} ="c" *+!R{^{ s_{3}}},
 (4,0) *\cir<2pt>{} ="d" *+!DL{^{ s_{4}}},
  (12,-12) *\cir<2pt>{} ="e" *+!L{^{ s_{5}}},
  (4,16)*+!{s},
 
\POS "a" \ar@{-}  "b",
\POS "a" \ar@{-}  "c",
\POS "a" \ar@{-}  "d",
\POS "a" \ar@{-}  "e",
\POS "b" \ar@{-}  "c",
\POS "c" \ar@{-}  "d",
\POS "d" \ar@{-}  "e", 
\end{xy}$}

\caption{The subset of $\mathbb{R}^{\geq 0}$ from which the value of $\gamma$ can be chosen so that $(K_{3,3}, l)$ is realizable is four disjoint intervals}
\label{K13b}
\end{center}
\end{figure}
\end{ex}

This analysis of the $(K_{3,3}, l)$ case is now distilled into Theorem \ref{k33thm}.

\begin{thm}\label{k33thm}
Given the weighted graph $(G_{0}, l_{ G_{0}})$, then the subset of $\mathbb{R}^{\geq 0}$ from which the value of $f=l_{G_{1}}(v_{5}v_{6})$ can be chosen so that $(G_{1}, l_{ G_{1}})$ is realizable is an interval; having chosen $l_{ G_{1}}$ and hence fixed $f$, then the subset of $\mathbb{R}^{\geq 0}$ from which the value of $\alpha=l_{G_{2}}(v_{1}v_{4})$ can be chosen so that $(G_{2}, l_{ G_{2}})$ is realizable is an interval; having chosen $l_{ G_{2}}$ and hence fixed $\alpha$, then the subset of $\mathbb{R}^{\geq 0}$ from which the value of $\beta=l_{G_{3}}(v_{3}v_{6})$ can be chosen so that $(G_{3}, l_{ G_{3}})$ is realizable is either an interval or the disjoint union of two intervals; finally, having chosen $l_{ G_{3}}$ and hence fixed $\beta$, then the subset of $\mathbb{R}^{\geq 0}$ from which the value of $\gamma=l(v_{2}v_{5})$ can be chosen so that $(K_{3,3}, l)$ is realizable is an interval or the disjoint union of two, three or four intervals.
\end{thm}

\section{The Moduli Space $M(K_{3,3}, l)$ }
A nice corollary of the analysis of previous section is that it is possible to establish a result relating to the connectedness of the moduli space $M(K_{3,3},l)$.

\begin{lem}Given the weighted graph $(K_{3,3}, l)$ then the moduli space  $M(K_{3,3}, l)$ can only have one, two, four, six or eight connected components.
\end{lem}
\begin{proof} As outlined above in relation to the weighted graph $(H, h)$ where $H$ contains three $3$-cycles, the moduli space $M(K_{3,3}, l)$ can contain at most eight connected components. If $(K_{3,3},l)$ is not realizable then the moduli space $M(K_{3,3},l)$ is the empty set and has, by definition, a single component. It is now required to show that the moduli space $M(K_{3,3},l)$ cannot contain three, five or seven connected components. In a similar fashion to the proof of Lemma \ref{varphi}, let the image of $\varphi_{v_{4}}|_{W_{K_{3,3},l,v_{1}v_{4}}(v_{i})}$ be denoted $W(v_{i})$ for $i\in\{2,3,5,6\}$. Clearly if a $W(v_{i})$ is connected then this does not imply that the corresponding moduli space is connected. However, if all $W(v_{i})$, for $i\in\{2,3,5,6\}$, are connected then the corresponding moduli space must be connected. It follows that when $M(K_{3,3},l)$ is disconnected, there must exist at least one $W(v_{i})$, for $i\in\{2,3,5,6\}$, which is disconnected.   \\

Each $W(v_{i})$, for $i\in\{2,3,5,6\}$, is symmetric about the $x$-axis i.e. $w\in W(v_{i})\iff \rho_{x}(w)\in W(v_{i})$, where $\rho_{x}$ is the reflection in the $x$-axis. This means that if $M(K_{3,3},l)$ is disconnected then there exists some disconnected $W(v_{i})$, for $i\in\{2,3,5,6\}$, such that the images of the realizations contained in the fibres of $\pi:M(K_{3,3},l)\to W(v_{i})$ over $W(v_{i})^{+}\subset W(v_{i})$ are all reflections in the $x$-axis of the images of realizations contained in the fibres of $\pi$ over $W(v_{i})^{-}\subset W(v_{i})$. Hence, if $M(K_{3,3},l)$ is disconnected then the connected components of $M(K_{3,3},l)$ must occur in pairs where the images of realizations contained in these components differ by a reflection in the $x$-axis. As $M(K_{3,3},l)$ can have at most eight connected components, and as the empty set has one connected component, then $M(K_{3,3},l)$ cannot contain three, five or seven connected components.\\

In Example \ref{ex1} the moduli space $M(K_{3,3}, l)$ is homeomorphic to the moduli space of an equilateral $4$-cycle $(C,l_{C})$. The moduli space $M(C,l_{C})$ is well known to be a connected space, see \cite{SV}, hence the moduli space $M(K_{3,3},l)$ can have one (non-empty) component. In Example \ref{ex2} there are three realizations $q,r$ and $s$ in which the length of $\gamma$ is contained in the interval $[1-\delta_{1},1+\delta_{1}]$, with $\delta_{1}\ll 1$, such that there does not exist a continuous deformation between any two of the realizations $q,r$ and $s$. This means that such a choice of $\gamma$ results in the moduli space $M(K_{3,3}, l)$ containing six connected components. In the same example, choosing $\gamma$ to be contained in the interval $[\sqrt{5}-\delta_{2},\sqrt{5}+\delta_{2}]$, where $\delta_{2}\ll 1$,  results in the moduli space $M(K_{3,3}, l)$ containing two connected components. In Example \ref{ex3} there does not exist a continuous deformation between realizations $p$ and $r$. The length of $\gamma$ is contained in the interval $[\frac{\sqrt{5}}{4}-\delta_{3},\frac{\sqrt{5}}{4}+\delta_{3}]$, where $\delta_{3}\ll 1$, and results in the moduli space $M(K_{3,3}, l)$ containing four connected components. Finally, Example \ref{ex5}, below, illustrates that there exists a scenario where it is possible to choose a value for $\gamma$ which results in the moduli space $M(K_{3,3}, l)$ containing eight connected components. This completes the proof.  
\end{proof}

\begin{ex}\label{ex5} Suppose that $a=\epsilon \ll b=f=\alpha=\beta=1$ and $c=d=e=\sqrt{2}$ then the subset of $\mathbb{R}^{\geq 0}$ from which the value of $\gamma$ can be chosen so that $(K_{3,3},l)$ is realizable is the interval $[\sqrt{2}-\delta,\sqrt{2}+\delta]$ where $\delta\ll 1$. Observe that the moduli space $M(K_{3,3}, l)$ has eight connected components. Recall that  $M(K_{3,3}, l)$ is homeomorphic to $C_{v_{4},v_{1}}(K_{3,3}, l)$. The images of realizations $p,q,r$ and $s$ which are each contained in distinct connected components of $C_{v_{4},v_{1}}(K_{3,3}, l)$ are shown in Fig. \ref{K13f}. Again, the larger nodes labeled $^{\; p_{6}}_{\; p_{1}}$, $^{\; q_{6}}_{\; q_{1}}$, $^{\; r_{6}}_{\; r_{1}}$ and $^{\; s_{6}}_{\; s_{1}}$ are analogous to the larger node contained in Fig. \ref{K13a}. The reflection $\rho_{x}$ in the $x$-axis applied to the images of each of the realizations $p,q,r$ and $s$ results in the images of the four realizations $\rho p, \rho q, \rho r$ and $\rho s$ which are each contained in one of the remaining four distinct  connected components of $C_{v_{4},v_{1}}(K_{3,3}, l)$.   

\begin{figure}[H]
\begin{center}
\scalebox{1}{$\begin{xy}
\POS (12,0) *\cir<4pt>{} ="a" *+!L{^{\; p_{6}}_{\; p_{1}}} ,
 (24,0) *\cir<2pt>{} ="b" *+!L{^{ p_{2}}},
 (12,-12) *\cir<2pt>{} ="c" *+!L{^{ p_{3}}},
 (0,0) *\cir<2pt>{} ="d" *+!R{^{ p_{4}}},
 (12,12) *\cir<2pt>{} ="e" *+!L{^{ p_{5}}},
  (0,15)*+!{p},
 
\POS "a" \ar@{-}  "b",
\POS "a" \ar@{-}  "c",
\POS "a" \ar@{-}  "d",
\POS "a" \ar@{-}  "e",
\POS "b" \ar@{-}  "c",
\POS "c" \ar@{-}  "d",
\POS "d" \ar@{-}  "e",
\POS "b" \ar@{-}_{\gamma}  "e", 
\end{xy}$
$\begin{xy}
\POS (12,0) *\cir<4pt>{} ="a" *+!L{^{\; q_{6}}_{\; q_{1}}} ,
 (24,0) *\cir<2pt>{} ="b" *+!L{^{ q_{2}}},
 (12,12) *\cir<2pt>{} ="c" *+!L{_{ q_{3}}^{ q_{5}}},
 (0,0) *\cir<2pt>{} ="d" *+!R{^{ q_{4}}},
 (12,12) *\cir<2pt>{} ="e" ,
  (0,15)*+!{q},
 
\POS "a" \ar@{-}  "b",
\POS "a" \ar@{-}  "c",
\POS "a" \ar@{-}  "d",
\POS "a" \ar@{-}  "e",
\POS "b" \ar@{-}  "c",
\POS "c" \ar@{-}  "d",
\POS "d" \ar@{-}  "e", 
\POS "b" \ar@{-}_{\gamma}  "e",
\end{xy}$
$\begin{xy}
\POS (12,0) *\cir<4pt>{} ="a" *+!L{^{\; r_{6}}_{\; r_{1}}} ,
 (0,0) *\cir<2pt>{} ="b" *+!R{_{ r_{2}}^{ r_{4}}},
 (12,-12) *\cir<2pt>{} ="c" *+!L{^{ r_{3}}},
 (0,0) *\cir<2pt>{} ="d",
 (12,12) *\cir<2pt>{} ="e" *+!L{^{ r_{5}}},
  (0,15)*+!{r},
 
\POS "a" \ar@{-}  "b",
\POS "a" \ar@{-}  "c",
\POS "a" \ar@{-}  "d",
\POS "a" \ar@{-}  "e",
\POS "b" \ar@{-}  "c",
\POS "c" \ar@{-}  "d",
\POS "d" \ar@{-}  "e", 
\POS "b" \ar@{-}^{\gamma}  "e",
\end{xy}$
$\begin{xy}
\POS (12,0) *\cir<4pt>{} ="a" *+!L{^{\; s_{6}}_{\; s_{1}}} ,
 (0,0) *\cir<2pt>{} ="b" *+!R{_{ s_{2}}^{ s_{4}}},
 (12,12) *\cir<2pt>{} ="c" *+!L{_{ s_{3}}^{ s_{5}}},
 (0,0) *\cir<2pt>{} ="d" ,
  (12,12) *\cir<2pt>{} ="e",
  (0,15)*+!{s},
 
\POS "a" \ar@{-}  "b",
\POS "a" \ar@{-}  "c",
\POS "a" \ar@{-}  "d",
\POS "a" \ar@{-}  "e",
\POS "b" \ar@{-}  "c",
\POS "c" \ar@{-}  "d",
\POS "d" \ar@{-}  "e", 
\POS "b" \ar@{-}^{\gamma}  "e",
\end{xy}$}

\caption{$M(K_{3,3}, l)$ can have eight connected components}
\label{K13f}
\end{center}
\end{figure}
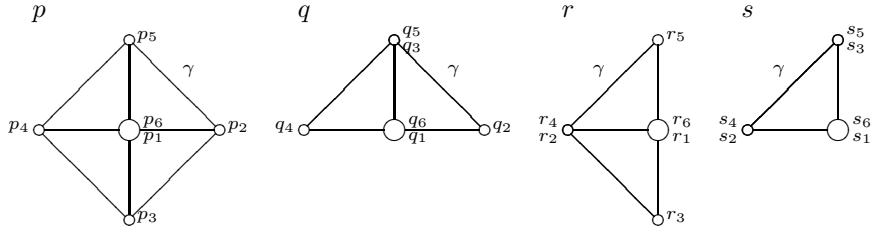
\end{ex}

\section{Moduli Spaces of Weighted Cyclic Subgraphs}
\hrule \vspace{0.5cm}

The moduli space of a weighted cycle is a well understood object, see for example \cite{F},\cite{Ha}, \cite{MT} and \cite{SV}. It may seem reasonable therefore  that whenever a weighted graph $(G,l)$ contains weighted cycles that by determining realizability and/or connectedness results for certain weighted cyclic subgraphs of $(G,l)$ then these results may be extended to realizability and/or connectedness results relating to the weighted graph $(G,l)$. This section contains two examples which show that properties of a moduli space $M(C,l_{\mid C})$ are not necessarily possessed by the moduli space of $M(G,l)$ where $C$ is a cyclic subgraph of $G$. 

\subsection{Realizability}
This section contains an example which shows that even though all weighted cyclic subgraphs of a given $(G,l)$ are realizable, the weighted graph $(G,l)$ may not itself be realizable.  

\begin{ex}\label{FourEx2}
Consider the graph $(G,l)$ where $V_{G}=\{v_{1},v_{2},v_{3},v_{4},v_{5},v_{6},$ $v_{7}\}$, $E_{G}=\{v_{1}v_{4},v_{1}v_{5},v_{1}v_{6},v_{2}v_{4},v_{2}v_{5}, v_{2}v_{7},v_{3}v_{4},v_{3}v_{6},v_{3}v_{7}\}$ and $l$ assigns the lengths $l(v_{1}v_{4})=2$, $l(v_{1}v_{5})=l(v_{1}v_{6})=4$, $l(v_{2}v_{4})=l(v_{3}v_{4})=\sqrt{13}$, $l(v_{2}v_{5})=l(v_{3}v_{6})=1$, $l(v_{3}v_{7})=\frac{7}{2}$ and $l(v_{2}v_{7})= \frac{1}{2}$. Observe that such a length assignment results in the situation where the weighted graph $(G,l)$ is not realizable and so $M(G,l)$ is empty. Further justification of the fact that $M(G,l)$ is empty can be found in Fig. \ref{Four8} which contains the weighted graph $(G,l)$ and two ``attempted realizations" of $(G,l)$ which are labeled $``p"$ and $``q"$. 

\begin{figure}[H]
\begin{center}
\scalebox{0.95}{$\begin{xy}
 \POS 
 (0,9) *\cir<2pt>{} ="a" ,
 (12,16) *\cir<2pt>{} ="b" ,
 (12,31) *\cir<2pt>{} ="c" ,
 (24,9) *\cir<2pt>{} ="d" ,
 (24,21) *\cir<2pt>{} ="e" ,
 (0,21) *\cir<2pt>{} ="f" ,
 (12,3) *\cir<2pt>{} ="g" ,
    (48,18) *\cir<2pt>{} ="a1" ,
 (60,18) *\cir<2pt>{} ="b1" ,
 (72,36) *\cir<2pt>{} ="c1" ,
 (72,0) *\cir<2pt>{} ="d1" ,
 (72,12) *\cir<2pt>{} ="d1ex" ,
 (72,15) *\cir<2pt>{} ="e1" ,
 (68,33) *\cir<2pt>{} ="f1" ,
 (68,3) *\cir<2pt>{} ="g1" ,
     (84,6) *\cir<2pt>{} ="a2" ,
 (96,6) *\cir<2pt>{} ="b2" ,
 (108,24) *\cir<2pt>{} ="c2" ,
 (96,30) *\cir<2pt>{} ="d2" ,
 (90,33) *\cir<2pt>{} ="d2ex" ,
 (90,36) *\cir<2pt>{} ="e2" ,
 (104,22) *\cir<2pt>{} ="f2" ,
 (99,27) *\cir<2pt>{} ="g2" ,
     (0,9) *+!UR{v_{2}} ,
 (12,16.5) *+!L{v_{4}} ,
 (12,31) *+!D{v_{1}} ,
 (24,9) *+!UL{v_{3}} ,
 (24,21) *+!DL{v_{6}} ,
 (0,21) *+!DR{v_{5}} ,
 (12,3) *+!U{v_{7}} ,
    (48,18) *+!U{p_{1}} ,
 (59,18) *+!U{p_{4}} ,
 (72.5,36.5) *+!L{p_{3}} ,
 (72.5,-1) *+!L{p_{2}} ,
 (72,11) *+!L{p_{2'}} ,
 (72,15) *+!L{p_{7}} ,
 (67,33) *+!D{p_{6}} ,
 (68,3) *+!U{p_{5}} ,
     (84,6) *+!U{q_{1}} ,
 (96,6) *+!U{q_{4}} ,
 (108,24) *+!L{q_{3}} ,
 (98.5,30) *+!D{q_{2}} ,
 (90,33) *+!R{q_{2'}} ,
 (90,36) *+!D{q_{7}} ,
 (104,22.5) *+!UL{q_{6}} ,
 (99,25.5) *+!L{q_{5}} ,
 (0,32) *+!{(G,l)} ,
 (48,36.5) *+!{``p"} ,
 (108,36.5) *+!{``q"} ,

  \POS "a" \ar@{-}|{\sqrt{13}} "b",
  \POS "b" \ar@{-}_{2} "c",
  \POS "b" \ar@{-}|{\sqrt{13}} "d",
  \POS "c" \ar@{-}^{4} "e",
  \POS "e" \ar@{-}^{1} "d",
  \POS "a" \ar@{-}_{\frac{1}{2}} "g",
  \POS "g" \ar@{-}_{\frac{7}{2}} "d",
  \POS "f" \ar@{-}^{4} "c",
  \POS "f" \ar@{-}_{1} "a",

  \POS "a1" \ar@{-} "b1",
  \POS "b1" \ar@{-} "c1",
  \POS "b1" \ar@{-} "d1",
  \POS "c1" \ar@{-} "e1",
  \POS "e1" \ar@{-} "d1ex",
  \POS "a1" \ar@{-} "g1",
  \POS "g1" \ar@{-} "d1",
  \POS "f1" \ar@{-} "c1",
  \POS "f1" \ar@{-} "a1",

  \POS "a2" \ar@{-} "b2",
  \POS "b2" \ar@{-} "c2",
  \POS "b2" \ar@{-} "d2",
  \POS "c2" \ar@{-} "e2",
  \POS "e2" \ar@{-} "d2ex",
  \POS "a2" \ar@{-} "g2",
  \POS "g2" \ar@{-} "d2",
  \POS "f2" \ar@{-} "c2",
  \POS "f2" \ar@{-} "a2",
  
  \POS "d1" \ar@{<.>}_{L} "d1ex",
  \POS "d2" \ar@{<.>}^{L} "d2ex",

 \end{xy}$}
\caption{The weighted graph $(G,l)$ and two ``attempted realizations" of $(G,l)$ which are labeled $``p"$ and $``q"$  }
\label{Four8}
\end{center}
\end{figure}
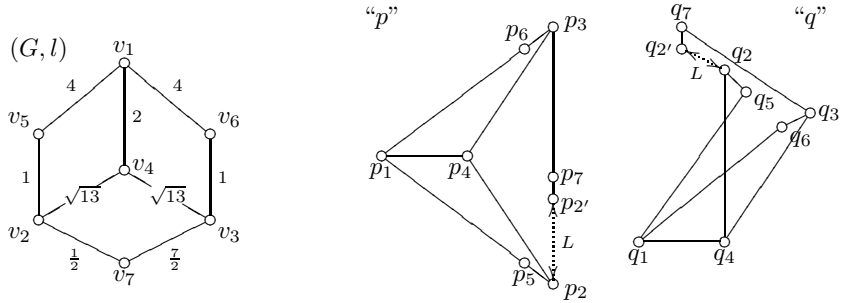

Observe that $G$ contains seven cyclic subgraphs and that 
all seven of these weighted cyclic subgraphs contained in $(G,l)$ are realizable. However, as the value of $L=d(p_{2},p_{2'})$ is always strictly positive in any ``attempted realization" of $(G,l)$, for example $``p"$ and $``q"$ in Fig. \ref{Four8}, then $(G,l)$ is not realizable.
\end{ex}

\subsection{Connectedness}

This section contains an example which shows that even though the moduli spaces of some weighted cyclic subgraphs of a given $(G,l)$ are not connected, the moduli space $M(G,l)$ may itself be connected.

\begin{ex}\label{3connex}

Consider the graph $(H,h)$ where $V_{H}=\{v_{1},v_{2},v_{3},v_{4},v_{5},v_{6},$ $v_{7}\}$, $E_{H}=\{v_{1}v_{4},v_{1}v_{5},v_{1}v_{6},v_{2}v_{4},v_{2}v_{5}, v_{2}v_{7},v_{3}v_{4},v_{3}v_{6},v_{3}v_{7}\}$ and $h$ assigns the lengths $h(v_{1}v_{4})=h(v_{2}v_{4})=h(v_{3}v_{4})=\sqrt{3}$ and  $h(v_{1}v_{5})=h(v_{1}v_{6})=h(v_{2}v_{5})=h(v_{2}v_{7})=h(v_{3}v_{6})=h(v_{3}v_{7})=\frac{3}{2}$ as illustrated in Fig. \ref{Five3}. \\

\begin{figure}[H]
\begin{center}
$\begin{xy}
 \POS (12,12) *\cir<2pt>{} ="a"   *+!L{v_{4}},
 (12,25) *\cir<2pt>{} ="b" *+!D{v_{1}},
 (23,17) *\cir<2pt>{} ="d" *+!L{v_{6}},
  (23,4) *\cir<2pt>{} ="f"  *+!UL{v_{3}} ,
 (12,-0.5) *\cir<2pt>{} ="g" *+!U{v_{7}},
  (1,4) *\cir<2pt>{} ="i" *+!UR{v_{2}},
 (1,17) *\cir<2pt>{} ="k" *+!R{v_{5}},
  (-5,28)  *+!{(H,h)} ,

\POS "a" \ar@{-}|{\sqrt{3}} "b",
\POS "a" \ar@{-}|{\sqrt{3}} "f",
\POS "a" \ar@{-}|{\sqrt{3}} "i",

\POS "d" \ar@{-}_{\frac{3}{2}}  "b",
\POS "d" \ar@{-}^{\frac{3}{2}}  "f",
\POS "g" \ar@{-}^{\frac{3}{2}}   "i",
\POS "g" \ar@{-}_{\frac{3}{2}}  "f",
\POS "k" \ar@{-}^{\frac{3}{2}}  "b",
\POS "k" \ar@{-}_{\frac{3}{2}}  "i",
\end{xy}$\hspace{1cm} \scalebox{0.9}{$\begin{xy}
 \POS (0,12) *\cir<2pt>{} ="v1"  *+!R{p_{1}},
 (18,12) *\cir<2pt>{} ="v2" *+!L{p_{4}} ,
 (28,28) *\cir<2pt>{} ="v3"  *+!L{p_{2}} ,
 (28,-4) *\cir<2pt>{} ="v4"  *+!L{p_{3}},
 (14,20) *\cir<2pt>{} ="v5" *+!D{p_{5}},
 (14,4) *\cir<2pt>{} ="v6" *+!U{p_{6}},
 (28,12) *\cir<2pt>{} ="v7" *+!L{p_{7}},
 (0,32) *+!{p}, 
(50,12) *\cir<2pt>{} ="vv1"  *+!R{q_{1}},
 (68,12) *\cir<2pt>{} ="vv2" *+!L{q_{4}} ,
 (58,-4) *\cir<2pt>{} ="vv3"  *+!L{q_{3}} ,
 (58,28) *\cir<2pt>{} ="vv4"  *+!L{q_{2}},
 (44,-2) *\cir<2pt>{} ="vv5" *+!R{q_{6}},
 (44,26) *\cir<2pt>{} ="vv6" *+!R{q_{5}},
 (58,12) *\cir<2pt>{} ="vv7" *+!DL{q_{7}},
 (38,32) *+!{q}, 
 
\POS "v2" \ar@{-} "v1",
\POS "v2" \ar@{-} "v3",
\POS "v2" \ar@{-} "v4",
\POS "v1" \ar@{-} "v5"
\POS "v5" \ar@{-} "v3"
\POS "v1" \ar@{-} "v6"
\POS "v6" \ar@{-} "v4"
\POS "v3" \ar@{-} "v7"
\POS "v7" \ar@{-} "v4"

\POS "vv2" \ar@{-} "vv1",
\POS "vv2" \ar@{-} "vv3",
\POS "vv2" \ar@{-} "vv4",
\POS "vv1" \ar@{-} "vv5"
\POS "vv5" \ar@{-} "vv3"
\POS "vv1" \ar@{-} "vv6"
\POS "vv6" \ar@{-} "vv4"
\POS "vv3" \ar@{-} "vv7"
\POS "vv7" \ar@{-} "vv4"

 \end{xy}$}

\caption{The weighted graph $(H,h)$ and the images of two realizations $p$ and $q$ of $(H,h)$ (the image of $\rho p$ is the reflection of the image of $p$ in the line containing $p_{1}$ and $p_{4}$)}
\label{Five3}
\end{center}
\end{figure}
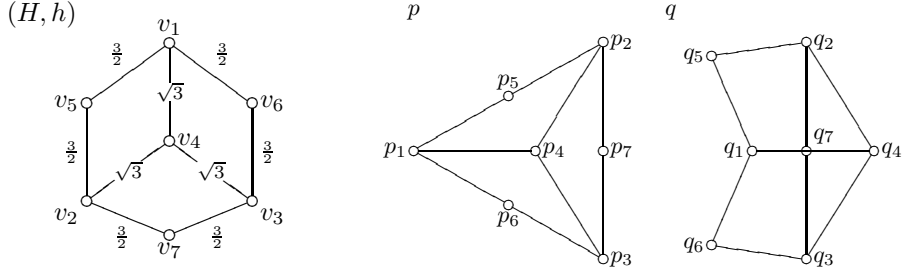
Note that there exists a realization $\rho p$ in the moduli space $M(H,h)$ whose image is a reflection of the image of $p$ in the half-line containing the images $p_{1}$ and $p_{4}$. Given $p,\rho p,q\in M(H,h)$, whose images are illustrated in Fig. \ref{Five3}, then there does not exist a path $\alpha_{1}:[0,1]\to M(H,h)$ such that $\alpha_{1}(0)=p$ and $\alpha_{1}(1)=\rho p$, a path $\alpha_{2}:[0,1]\to M(H,h)$ such that $\alpha_{2}(0)=p$ and $\alpha_{2}(1)=q$,or a path $\alpha_{3}:[0,1]\to M(H,h)$ such that $\alpha_{3}(0)=q$ and $\alpha_{3}(1)=\rho p$. However, observe that there does exist a path $\beta:[0,1]\to M(H,h)$ such that $\beta(0)=q$ and $\beta(1)=\rho q$, where $\rho q$ is the realization whose image is the reflection of the image of $q$ in the line containing $q_{1}$ and $q_{4}$. It follows that the moduli space $M(H,h)$ has three connected components.\\

Given $(H,h)$ as per Fig. \ref{Five3}, then consider the weighted graph $(G,l)$ where $G$ has vertex set $V_{G}=V_{H}\cup \{v_{8}\}$, edge set $E_{G}=E_{H}\cup \{v_{2}v_{8},v_{3}v_{8}\}$ and $l$ is an extension of $h$ which also assigns the lengths $l(v_{2}v_{8})=l(v_{3}v_{8})=\sqrt{2}$.\\

\begin{figure}[t]
\begin{center}
$\begin{xy}
 \POS (12,15) *\cir<2pt>{} ="a"   *+!L{v_{4}},
 (12,28) *\cir<2pt>{} ="b" *+!D{v_{1}},
 (23,17) *\cir<2pt>{} ="d" *+!L{v_{6}},
  (23,4) *\cir<2pt>{} ="f"  *+!UL{v_{3}} ,
 (12,-4) *\cir<2pt>{} ="g" *+!U{v_{7}},
  (1,4) *\cir<2pt>{} ="i" *+!UR{v_{2}},
 (1,17) *\cir<2pt>{} ="k" *+!R{v_{5}},
 (12,4) *\cir<2pt>{} ="j" *+!U{v_{8}},
  (-5,28)  *+!{(G,l)} ,

\POS "a" \ar@{-}|{\sqrt{3}} "b",
\POS "a" \ar@{-}|{\sqrt{3}} "f",
\POS "a" \ar@{-}|{\sqrt{3}} "i",
\POS "i" \ar@{-}|{\sqrt{2}} "j",
\POS "j" \ar@{-}|{\sqrt{2}} "f",

\POS "d" \ar@{-}_{\frac{3}{2}}  "b",
\POS "d" \ar@{-}^{\frac{3}{2}}  "f",
\POS "g" \ar@{-}^{\frac{3}{2}}   "i",
\POS "g" \ar@{-}_{\frac{3}{2}}  "f",
\POS "k" \ar@{-}^{\frac{3}{2}}  "b",
\POS "k" \ar@{-}_{\frac{3}{2}}  "i",
\end{xy}$\hspace{1cm} \scalebox{0.9}{$\begin{xy}
 \POS (0,12) *\cir<2pt>{} ="vv1" ,
 (1,9) *+!{r_{1}},
 (18,12) *\cir<2pt>{} ="vv2" *+!L{r_{4}} ,
 (4,0) *\cir<2pt>{} ="vv3"  *+!L{r_{3}} ,
 (4,24) *\cir<2pt>{} ="vv4"  *+!L{r_{2}},
 (-12,0) *\cir<2pt>{} ="vv5" *+!R{r_{6}},
 (-12,24) *\cir<2pt>{} ="vv6" *+!R{r_{5}},
 (-3,12) *\cir<2pt>{} ="vv7" *+!R{r_{7}},
 (4,12) *\cir<2pt>{} ="vv8" *+!UL{r_{8}},
 (-12,32) *+!{r}, 

\POS "vv2" \ar@{-} "vv1",
\POS "vv2" \ar@{-} "vv3",
\POS "vv2" \ar@{-} "vv4",
\POS "vv1" \ar@{-} "vv5"
\POS "vv5" \ar@{-} "vv3"
\POS "vv1" \ar@{-} "vv6"
\POS "vv6" \ar@{-} "vv4"
\POS "vv3" \ar@{-} "vv7"
\POS "vv7" \ar@{-} "vv4"
\POS "vv3" \ar@{-} "vv8"
\POS "vv4" \ar@{-} "vv8"

 \end{xy}$}

\caption{The weighted graph $(G,l)$ and the image of a realization $r$ of the weighted graph $(G,l)$}
\label{Five33}
\end{center}
\end{figure}
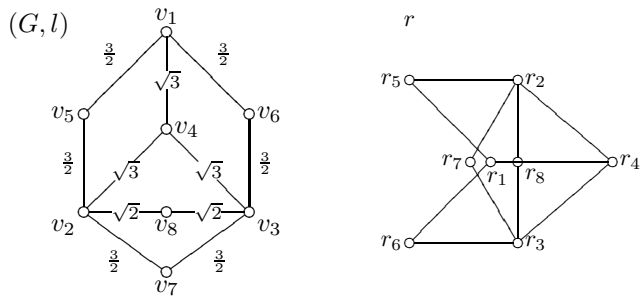

Consider now the inclusion map $\iota:M(G,l)\to M(H,h)$. Observe that $\iota$ is not surjective as neither $p$ nor $\rho p$, as per Fig. \ref{Five3}, are mapped onto by $\iota$. Note that the components of $M(H,l)$ which contain the realizations $p$ and $\rho p$ each contain just a single point. Note also that there exists a path $\gamma:[0,1]\to M(G,l)$ such that $\gamma(0)=r$ and $\gamma(1)=\rho r$ where $\rho r$ is the realization of $(G,l)$ whose image is the reflection of the image of $r$ in the line containing $r_{1}$ and $r_{4}$. It follows from these notes that $M(G,l)$ is connected.

\end{ex}


\subsection*{Acknowledgments}
The author would like to thank Dr. James Cruickshank who acted as Ph.D advisor during the writing of \cite{JML}, in which this note constitutes a chapter. 


\begin{thebibliography}{1}


 \bibitem{BC}
    {Belk, Maria and Connelly, Robert},
      {{Realizability of graphs}},
   {\it Discrete Comput. Geom.},
    {\bf 37}(2), {(2007)}, {125--137},
      {MR2295049 (2007k:05147)},

\bibitem{JC}
    {Cruickshank, James and McLaughlin, Jonathan},
      {{Series Parallel Linkages}},
   {\tt http://arxiv.org/abs/0911.5293 }, {(2009)},


 \bibitem{CS}
    {Curtis, Robyn and Steiner, Marcel},
      {{Configuration spaces of planar pentagons}},
   {\it Amer. Math. Monthly},
     {\bf 114}(3), {(2007)}, {183--201},
      {MR2290284 (2008a:52029)},

\bibitem{AS}
   {D'Andrea, K. and Sombra, M.},
     {{The {C}ayley-{M}enger determinant is irreducible for {$n\geq
              3$} }},
   {\it Sibirsk. Mat. Zh.},
      {\bf 46}(1), {(2005)}, {90--97},
      {MR2141304 (2006a:15019)},


 \bibitem{D}
    {Diestel, Reinhard},
      {{Graph theory}}, second edition
    {\it Graduate Texts in Mathematics},
    {\bf 173}, {Springer-Verlag}, {New York},
      {(2000)}, {MR1743598},
      
  \bibitem{F}
    {Farber, Michael},
      {{Invitation to topological robotics}},
    {\it Zurich Lectures in Advanced Mathematics},
  {European Mathematical Society (EMS), Z\"urich},
      {(2008)}, {MR2455573 (2010a:55018)},
 
 \bibitem{Ha}
     {Hausmann, J.-C. and Knutson, A.},
    {{The cohomology ring of polygon spaces}},
   {\it Ann. Inst. Fourier (Grenoble)},
    {\bf 48}(1), {(1998)}, {281--321}, {MR1614965 (99a:58027)},



  \bibitem{He}
    {Hendrickson, Bruce},
      {{The molecule problem: exploiting structure in global
              optimization}},
   {\it SIAM J. Optim.}, {\bf 5}(4), (1995), {835--857},
      {MR1358807 (96g:90093)},


 \bibitem{KM1}
    {Kapovich, Michael and Millson, John},
     {{On the moduli space of polygons in the {E}uclidean plane}},
   {\it J. Differential Geom.},
     {\bf 42}(2),  {(1995)}, {430--464},
      {MR1366551 (96k:58035)},
 
  \bibitem{JML}    
{McLaughlin, Jonathan}, 
{{Moduli spaces of planar realizations of weighted graphs}},
{\it Ph.D thesis, National University of Ireland, Galway},
(2009), 

 \bibitem{MT}
     {Milgram, R. James and Trinkle, J. C.},
     {{The geometry of configuration spaces for closed chains in two
              and three dimensions}},
   {\it Homology Homotopy Appl.},
    {\bf 6}(1), {(2004)},  {237--267},
     {MR2076003 (2005e:55026)},
     
  \bibitem{SV}
     {Shimamoto, Don and Vanderwaart, Catherine},
      {{Spaces of polygons in the plane and {M}orse theory}},
    {\it Amer. Math. Monthly},
    {\bf 112}(4), {(2005)}, {289--310},
     {MR2125272 (2005m:52032)}, 


 \bibitem{TW}    
{W.P. Thurston and J.R. Weeks}, 
{{The Mathematics of Three-Dimensional Manifolds}},
{\it Scientific American},
{\bf 251}, (1984),  {108--120},
 
 
 

   

  
\end{thebibliography}
\end{document}